\documentclass[a4paper]{amsart}

\usepackage[utf8]{inputenc}
\usepackage{enumerate}

\newcommand{\Nn}{\mathbb{N}}

\newcommand{\N}{\mathbb{N}}
\newcommand{\Z}{\mathbb{Z}}

\renewcommand {\epsilon}{\varepsilon}

\renewcommand {\leq}{\leqslant}
\renewcommand {\geq}{\geqslant}

\newcommand{\lcm}{\mathop{\mathrm{lcm}}\nolimits}

\renewcommand{\gcd}{\mathop{\mathrm{gcd}}\nolimits}

\newcommand{\tq}{\;\: | \:\:}
\newcommand{\mesp}{\;\:}

\newcommand{\absolue}[1]{\left| #1 \right|}
\newcommand {\vfi}{\varphi}

\newcommand{\iffl}{\longleftrightarrow}
\newcommand{\iffs}{\leftrightarrow}
\newtheorem{theorem}{Theorem}[section]
\newtheorem{lem}[theorem]{Lemma}
\newtheorem{proposition}[theorem]{Proposition}
\newtheorem{corr}[theorem]{Corollary}

\newtheorem{definition}[theorem]{Definition}
\newtheorem{defn}[theorem]{Definition}

\newtheorem{remark}[theorem]{Remark}
\newtheorem{fact}[theorem]{Fact}

\theoremstyle{definition}
\newtheorem{example}[theorem]{Example}

\newtheorem{nota}[theorem]{Notation}

\usepackage{amssymb}

\title{Valued modules over skew polynomial rings I}
\author{Gönenç Onay}

\address{Mathematics Department,
Faculty of Science and Letters
Mimar Sinan Fine Arts University
Bomonti Campus 34380
Istanbul}
\email{gonenc.onay@msgsu.edu.tr}

\thanks{The author would like to thank {\it Equipe de Logique Mathématiques de Paris} for supporting his stay during September 2013  where the current work is mainly completed. Françoise Delon read  this paper several times, made lot of valuable comments, to whom the author is  very much grateful. Many thanks go to Piotr Kowalski for his final remarks.}

\begin{document}
\sloppy
\begin{abstract}

We introduce a notion of valued module which is suitable to study valued fields of positive characteristic. Then we built-up a robust theory of henselianity in the language of valued modules and prove  Ax-Kochen Ershov type results.
\end{abstract}
\maketitle

\section{Introduction}
A valued abelian group, is an abelian group $M$ together with a linearly ordered set $\Delta$ and a function
$v:M \to \Delta$, such that \begin{enumerate}[a.]
\item $v(x\pm y)\geq \min\{v(x),v(y)\}$
\item $v(x)=\infty \Leftrightarrow x=0$
\end{enumerate}
where $\infty$ is the maximum of $\Delta$.

A module such that the underlying abelian group is valued will be called a valued module. Depending on the context one wants to study, various compatibility assumptions relating the valuation $v$ and the action of scalars are considered in the literature.  The first article on the subject that the author is aware of is due to Fleischer (\cite{fleischer}); it extends the notion of
Krull valuation to modules where $v(x.r)\geq v(x)$ for any scalar $r$ and proves the equivalence of maximality and (pseudo-) convergence of pseudo-Cauchy series. In Rohwer's thesis (\cite{rohwer}) modules are also used to understand the model theory of valued fields of positive characteristic or more generally valued  difference fields. In more recent articles of Point and B\'elair  \cite{BePo} (resp. Maalouf \cite{fares}) valued modules (resp. vector spaces) are studied.

Similarly as in \cite{rohwer}, we are interested  in  modules  which  come from the theory of valued fields in characteristic $p>0$. The rings that we will consider
are unitary and  arise as $K$-subalgebras of the endomorphisms ring of a field $K$ generated by a distinguished endomorphism $\vfi$ over $K$.
\begin{definition}\label{defR}
Let $K$ be a field and $\varphi$  a ring endomorphism of $K$. The ring $K[t;\vfi]$ is given as a set by formal sums $\sum_{i \in \Nn} t^ia_i$
where $t$ is the indeterminate, $a_i \in K$ and $\{i \in \Nn \tq a_i\neq 0\}$ is finite. The addition is defined term by term and the multiplication obeys the (non-)commutation rule: $at=t\vfi(a)$. More precisely, \begin{enumerate}
\item $\sum_{i \in \Nn} t^ia_i + \sum_{i \in \Nn} t^ib_i = \sum_{i \in \Nn} t^i(a_i + b_i) $
\item $\left(\sum_{i \in \Nn} t^ia_i\right)\left(\sum_{i \in \Nn} t^ib_i\right)=\sum_{i \in \Nn} t^ic_i$ where
$c_i=\sum_{k+l=i} \vfi^l(a_k)b_l$.
\end{enumerate}
\end{definition}
Most often for fixed $K$ and $\vfi$ we will refer to this ring $K[t;\varphi]$ as $R$. Any  field $K$  of characteristic  of $p>0$ has a natural module structure over its ring of additive polynomials: these are polynomials of the form $\sum_{i=0}^n a_iX^{p^{i}}$, where the ring operations are given by usual addition and (right-)composition. This ring is isomorphic to the ring $K[t; Frob]$ where $Frob$ is the the map $x \mapsto x^p$. It is widely believed that serious problems in the model theory of valued fields in positive characteristic arise from additive polynomials (see for example \cite{Kuhlmann06additivepolynomials}). Studying and axiomatizing the underlying ultrametric module structure can be considered as an abstract analysis that isolates this phenomenon. Note also that any valued field together with an arbitrary automorphism inherits an analogue module structure.

After the famous Ax-Kochen and Ershov (A-K, E) theorem which provides a corrected form of a conjecture of Artin\footnote{Artin conjectured that the field $\mathbb{Q}_p$ is $C_2$ (i.e.  every homogeneous polynomial of degree $d$ with $>d^2$ variables has a non trivial zero)  for any prime $p$. This is false by a result of Terjanian (see \cite{terjanian}).}, there were similar theorems  permitting to understand the first order theories of some other class of valued fields (see \cite{vdd}) for a survey). In this paper,  we establish (A-K, E) type results by proving a relative elimination of quantifiers  for modules called {\it henselian divisible valued  $R$-modules}. We first observe the theory of that divisible valued modules is not suitable to reach our aim (cf. counter example \ref{contre-exemple} and the remark which follows). Thus we add an extra assumption on the valuation, which comes from an interpretation of Hensel's lemma as a local inversion theorem. More precisely, in this context, there is an abstract analogue of the maximal ideal of a valuation ring; call this $M_{>\theta}$ and then our assumption asserts that the multiplication by some special elements of $R$ (we will call them {\it separable polynomials}) induces a bijection $M_{>\theta} \to M_{>\theta}$

Fix $R=K[K;\varphi]$ as above for given $\varphi$ and $K$. In our model theoretical setting valued $R$-modules are considered in the two sorted first order language
$$L:=\{0,+,-,(.r)_{r \in R}\}\cup \{<,\tau,\infty,(R_n)_{n \in \mathbb{N}}\} \cup \{v \}$$ where
\begin{enumerate}
\item the {\it module sort} language is $L_{Mod}(R):=\{0,+,-,\{.r\}_{r \in R}\}$
(which is the usual language of modules),
\item the {\it value set sort} language is $L_V:=\{<,\tau,\infty,(R_n)_{n \in
\mathbb{N}}\}$ with $\tau$  a unary function symbol and $R_n$  a unary
predicate symbol for each $n$,
\item $v$ is a unary function symbol to be interpreted as the valuation.
\end{enumerate}
\medskip
We  describe now the content of our paper.

In section 2, we consider pure $R$-modules and describe the completions of the theory of divisible $R$-modules as follows:

\begin{theorem}\label{completions} Any complete theory of non zero  divisible $R$-modules admits elimination of quantifiers and  is obtained by specifying for each irreducible $q \in R$, the number of elements annihilated by  $q$.
\end{theorem}

In the Section 3, we introduce  valued  $R$-modules $(M,\Delta,v)$  which by definition  satisfy the following compatibility properties: \begin{enumerate}
\item $v(x.\lambda)=v(x)$ for all unit $\lambda$ in $R$.
\item The map $x \to v(x.t)$ induces a strictly increasing function $\tau: \Delta \to \Delta$ and there exists at most one value $\theta \in \Delta$, such that, if  $r\in R\setminus\{0\}$ and  $x\in M$ with $v(x)\neq \theta$, then $v(x.r)=\tau^{k}(v(x))$ for some $k \in \mathbb{N}$ depending only on $r$ and $v(x)$.
\end{enumerate}

Note that any valued field $(U,v)$ of characteristic $p>0$ inherits naturally a valued $K[t;Frob]$-module structure where  $K\subseteq U$ is trivially valued by $v$ and  the function $\tau$ is  $\gamma \mapsto p\gamma$.

We then investigate the $L$-theory of {\it henselian} and {\it henselian divisible  $R$-modules}. We have the following main theorem which  enables us to recover the theory of a divisible henselian valued $R$-module from the theory of its value set (in $L_V$) and from the theory of its  torsion submodule (in $L_{Mod}(R)$) which, in some sense, plays the role of the \emph{residue field}.

\begin{theorem}[A-K,E $ \equiv $]\label{aketrivial}
Let $ (F, v) $ and $ (G, w) $ be two non zero henselian divisible modules such that $F_{tor}$ and $G_{tor}$ are elementary equivalent as $L_{Mod}(R)$-structures and $ v (F) $ and $ v (G) $ are elementary equivalent in the language $ L_{V} $. Then $ (F, v) $ and $ (G, v) $ are elementary equivalent as $ L $-structures.
\end{theorem}

This article corresponds to Chapters Two and Three of  author's thesis (cf. \cite{gonenc}).

\section{Divisible $R$-modules}\label{sec:divmod}
First, we  quickly summarize some basic algebraic facts about the rings over which we will consider our modules. Once can refer to \cite{cohn} section 2.1 for more details.

Let  $K$ be a field and $\varphi$  a {\it field endomorphism} of  $K$. Set
$R:=K[t;\vfi]$, the ring of $\vfi$-twisted polynomials in the variable $t$ (see Definition \ref{defR} in Introduction). We will mean by {\bf polynomial} any element of this ring. Each $r \in R$ can be written uniquely as $\sum_{i \in \mathbb{N}} t^ia_i$, with $I:=\{i \tq a_i\neq 0\}$ finite. The ring $R$ is equipped with the degree function  defined by $\deg(r)=\max I$ (with the convention that $\max  \emptyset = -\infty$). We have $\deg(pq)=\deg(qp)=\deg(p) + \deg(q)$. Note that $R$ is a domain, i.e. has no zero divisors. The set of units of $R$  is  $K^{\times}$ and together with the function $\deg$,  $R$ is right euclidean  (and hence right principal)\footnote{ Note that this ring is also left euclidean whenever $\varphi$ is an automorphism.} : for all $q,q'\in R$ with $q'\neq 0$ there exists $r,r'\in R$ such that $q=q'r' + r$ where $\deg(r)<deg(q')$ and $r$ is unique up to right multiplication by a unit. We say that $q'$ divides $q$ if $q$ is a right multiple of $q'$, that is if $r=0$ above. For any  non zero  $a$ and  $b$
there exists a unique monic common divisor of maximal degree,   denoted by $\gcd(a,b)$.

Furthermore, by the commutation rule $at=t\varphi(a)$, any non zero element $q$ can be written as
\begin{equation}\label{eqirr}q=t^nq_1 \dots q_s
\end{equation}
where  $t$ and the $q_i$ are irreducible.
Note  that $R$ is in particular a right Ore ring, that is every two non zero elements $a$ and $b$ have a non zero common right multiple. Taking into account the degree function there is a common right multiple of least degree unique up to a unit: we denote it  $\lcm(a,b)$ as usual. As any Ore ring, $R$ is embedded into a smallest  (right-)division ring, called its skew field of fractions.

\begin{nota} We will denote by $D$  the skew field of fractions of $R$. \end{nota}

Any  $R$-module that we will consider is a right $R$-module, and we will simply say $R$-module.

The following result is an easy observation as in the case of a torsion free divisible abelian groups:

\begin{lem}
	Any torsion free divisible module over a Ore ring $S$ has a unique vector space structure over the skew field of  fraction of $S$. In particular, every divisible torsion free $R$-module is a $D$-vector space.
\end{lem}

\begin{nota}
Let $ M $ be an $ R $-module. We denote by $ M_{tor} $ the set
    $$ \{x \in M \tq  \exists r \in R \setminus \{0 \} \mesp \text{such that} \mesp x.r = 0 \}.$$ \end{nota}
\begin{lem}\label{Mtordivisible}
If $ M $ is a  divisible $ R $-module then $ M_{tor} $ is a divisible submodule of $M$.
\begin{proof}
Since $ R $ is a  Ore ring, $ M_{tor} $ is a submodule (cf. \cite{DDP1} proposition 3.5). Moreover, if $ 0 \neq x \in M_{tor} $, with $ q \in R \setminus \{0 \} $ such that $ x.q = 0 $ then, for all $ r \in R \setminus \{0 \} $ and for all $ y \in M $ such that $ y.r = x $, we have
$y.rq  =  0 $. Therefore $ y \in M_{tor} $, which shows that $ M_{tor} $ is divisible.
 \end{proof}
\end{lem}

\begin{remark}\label{rmq:divisiblemodulesareinfinite} Any torsion free  non zero divisible $R$-module is infinite since it is a $D$-vector space.

Furthermore if $M$ is a non zero  divisible $R$-module then  for every $r \neq 0$, $M.r=M\neq 0$. Hence  $M_{tor}.r \neq 0$ when  $r\neq 0$ and $M_{tor}\neq 0$. In
particular, $M_{tor}$ is infinite if it is non zero, since otherwise, the least common multiple of annihilators of non zero elements of $M_{tor}$ would annihilate  $M_{tor}$.
\end{remark}

Recall that a module $I$ over a ring $S$ is said to be injective if for any given inclusion of $S$-modules $N \subseteq M$, and any homomorphism $f:N \to I$, $f$ admits an extension to $M$.

In the following  lemma we state that in the case of $R$-modules the notions of injectivity and divisibility coincide, and we also summarize well-known facts about injective modules that we will frequently use in the rest of this paper.

\begin{lem} \label{modinjective} Let $S$ be any ring.
\begin{enumerate}
\item A divisible module over a right principal ring is injective, and an injective module over a domain
 is divisible. In particular for every $ R $-module $ M $, $ M $ is divisible if and only if it is injective.

\item Any injective $S$-module is a direct summand in every $S$-module which contains it.
\item  Any $S$-module $M$ has an extension $ N \supseteq M$ maximal with the property $$ X \cap M \neq 0  \mesp \text{for any non zero submodule} \mesp X \subseteq N.$$ $N$ is called an injective hull of $M$ and is unique up to an $M$-isomorphism of $S$-modules.

\item Let $ N \supseteq M $ be $R$-modules. If $N$ is an injective hull of $ M $ and $ x \in N \setminus \{0 \} $, then there exists $ r \in R$ such that $x.r \in M \setminus \{0 \}$.
\end{enumerate}
\begin{proof}
We refer to \cite{robinson} (p. 156-164). The last assertion is  Exercise 3. page 164.
 \end{proof}
\end{lem}

\begin{defn}
For  $R$-modules  $M \subseteq N$, where $N$ is divisible, we call divisible closure of $M$  any  injective hull of $M$ inside $N$.
\end{defn}

\begin{remark}
\begin{enumerate}
\item Let $M\subseteq E \subseteq N$ be $R$-modules, with $E$ and $N$ divisible. Then $E$ is a divisible closure of $M$, if and only if for all $x \in E$ there exists a non zero $r \in R$ such that $x.r\in E$. Moreover note that
$E$ can be different from
$$\{x \in N \tq \exists  r\in R, \, r\neq 0 \mesp \& \mesp x.r \in M\}.$$

\item By   lemma \ref{modinjective}, divisible closures of an $R$-module are isomorphic as pure modules.  When we will consider  valued modules   we will see that the various divisible closures are generally  neither isomorphic nor even elementary equivalent as valued modules (see Proposition \ref{contre-exemple} and Corollary \ref{rem-contre-ex}).
\end{enumerate}
\end{remark}

The following observation is essential:

\begin{lem}
  \label{lemqx} Let $M$ be an $R$-module, $A \subseteq M$ a submodule of $M$ and $x
  \in M \setminus A$ such that $x.r \in A$ for some $r \in R \setminus \{0\}$.
  Then there exists $q \in R$ of minimal degree such that $x.q \in A$. For
  such a $q$ and for all $r$ such that $x.r \in A$, $q$ divides $r$.

  \begin{proof}
    Let $I:= \{r \in R \mesp | \mesp
     x.r \in A\}$; $I$ is a right ideal,  therefore  principal.
    Hence there is a generator of $I$, that we can take as $q$.
  \end{proof}
\end{lem}

\begin{defn}
Any polynomial $q$ as above
 is called a minimal polynomial of $x$ over $A$. In
  particular, for all $x \in M_{tor} \setminus \{0\}$, there exists
  monic polynomials of minimal degree such that $x.q = 0$. It is called a
  minimal polynomial of $x$.
\end{defn}

\begin{remark}
A minimal polynomial of $x$ over $A$ divides any minimal polynomial of $x$. If
$x \in M\setminus A$ and $x.r \in A$  with $r$ irreducible then $r$ is a minimal polynomial
of $x$ over $A$. Minimal polynomials are unique  up to a unit.
\end{remark}
\begin{defn}
For any $r \in R$ and any $R$-module $M$,  the annihilator set of $r$ in $M$, is the
set $ann_M(r):=\{x \in M \tq x.r=0\}$ and its elements are called zeros or roots of $r$.
\end{defn}

\begin{nota} For an $R $-module $M $, let $\eta_M$ be the cardinal-valued function defined for all  $q\in R$,
 by $\eta_M (q) := \absolue{ann_M (q)} $.
\end{nota}
\begin{proposition}\label{torsionfini}
 Let $M $ and $N $ be divisible $R$-modules satisfying
$\eta_M  \leq  \eta_N $.  Suppose  $\eta_M(q)$ is finite for all non zero $q$ and $f:A\to N$ is an  $R$-module embedding where $A \subseteq M_{tor}$ is a submodule.  Then $f$ extends to an $R$-module embedding of  $M_{tor} $ into $N_{tor}$. In particular if $ ann_M (q)=ann_N (q)<\infty $ for all non zero $q \in R$,  then $M_{tor}$ and $N_{tor}$ are isomorphic.
\begin{proof}
Note that the range of $f$ is in $N_{tor}$. Let $x \in A\setminus M_{tor}$ and $q$ be a minimal polynomial of $x$ over $A$.
Write $q=q_1\dots q_k$ where the $q_i$ are irreducible. Since $q$ is a minimal polynomial of $x$ over $A$,
$z:=x.q_1\dots q_{k-1} \notin A$.  Set $r:=q_k$. Since $r$ is irreducible it is a minimal polynomial
of $z$ over $A$.

\noindent
{\it Claim} : There exists $y \in N_{tor}\setminus f(A)$ such that $y.r=f(z.r)$.

\noindent
{\it Proof of the Claim}.
Set $b:=f(z.r)$.
Since $N_{tor}$ is divisible it contains some $b_1$ such that $b_1.r=b$ Suppose $b_1 \in f(A)$.  Then $a_1:=f^{-1}(b_1) \in A$.
It follows that $(a_1-z).r=0$.  Note that $(a_1-z) \notin A$ and $r$ has the same number of zeros in $A$ and $f(A)$. Since $\eta_M(r)$ is finite and $\eta_M(r) \leq \eta_N(r) $ by hypothesis, there exists $b_0 \in N\setminus f(A)$ such that $b_0.r=0$. Now $b_0+b_1 \notin f(A)$ and $(b_0 +b_1).r=b$.$\dagger$

Take  such an $y$, then since $r$ is irreducible it is a minimal polynomial of $y$ over $A$. Now it is easy to check
that the map  $A+z.R \to f(A) + y.R$ sending $y$ to $z$ is an isomorphism of $R$-modules.

Using Zorn Lemma we can extend $f$ to $M_{tor}$.
\end{proof}
\end{proposition}

%\begin{corr} [of lemmas \ref{modinjective} and \ref{Mtordivisible}] \label{thM=thM_tor}
%The torsion submodule of a divisible $ R $-module is a direct summand and it is infinite if not $0$. As a consequence,  for a non zero $R$-module $ A $, $A \equiv A_{tor}$.
%\begin{proof}
%By lemma \ref{Mtordivisible}  the torsion submodule of an $R$-module $A$ is divisible, therefore it is injective by lemma \ref{modinjective}. Hence  it is a direct summand. Thus  one has $ A = A_{tor} \oplus A '$. In this case $ A '$ divisible and torsion free hence a $ D $-vector space. So they are infinite because $ R $ is infinite. Therefore they are elementary equivalent. The conclusion  follows by the Feferman-Vaught theorem (see e.g. \cite{hodges}  Section 9.6).
%\end{proof}
%\end{corr}

\subsection{Complete  theories of divisible $ R $-modules}
Let $S$ be a ring. Let  $L_{Mod}(S):=\{0,+,-\}\cup\{.s \tq s\in S\}$ be the language of (right-)$S$-modules. An equation  is an $L_{Mod}(S)$-formula
$\gamma(x_1,\dots,x_m)$ of the form
$$\gamma(\overline{x}):x_1.r_1 + \dots + x_m.r_m=0.$$ A primitive positive
(p.p.) formula $\phi(x_1,\dots,x_m)$ of the language of $S$-modules is a
formula of the form $$\exists \overline{y} \mesp
\left(\gamma_1(\overline{x},\overline{y})\wedge\dots\wedge\gamma_n
(\overline{x},\overline{y})\right)$$ where the $\gamma_i$ are equations. For instance,
by choosing $\gamma_1(x,y)= x.1 + (-y).q=0$ and $\gamma_2(x,y)=x.r + y.0=0$, the formula  $\phi(x):\exists y\mesp (y.q=x\wedge x.r=0)$ is a p.p. formula. Note also that a p.p.
formula can be written as  $$\exists y_1, \exists y_2, \dots \exists y_k \mesp
(y_1, \dots, y_k).A=(x_1,\dots x_n).B $$
where $A$ and $B$ are matrices with coefficients in $S$. Such formulas define subgroups, called p.p. subgroups.
It is a well known fact that modulo a complete  theory of $S$-modules  every formula is equivalent to a boolean combination of p.p. formulas.  See \cite{prest} for more details. Moreover:

\begin{proposition}\label{ppindexs} The complete theory of a $S$-module
$\langle M,+,-, 0, \{.s\}_{s \in S}\rangle$ is given by the p.p. indexes  (these are
 cardinalities in $\mathbb{N}\cup \{ \infty \}$ of relative quotients of p.p.  subgroups  of $M$).
\end{proposition}

In our case we have a simpler description of p.p. definable sets.
\begin{proposition} \label{pp2atom}
Modulo the theory of divisible $ R $-modules, any  p.p. formula is equivalent  to a conjunction of atomic formulas.
Consequently,  a p.p. formula with only one free variable is equivalent to one of the form $x.q=0$ for some $q \in R$.
\begin{proof}
Let $ \phi (x_1, \dots, x_n):= \exists y_1, \dots y_k \mesp \overline{y}. A = \overline{x}. B $ be a p.p. formula. By Proposition 6.1 of \cite{DDP1}
there exist invertible matrices $ P $ and $ Q $, where $ P $ has coefficients in $ \{0,1 \} $, such that
$ A': = PAQ $ is of the form $ (A_1, 0) $, where $ A_1 $ is a $ k \times l $ lower triangular matrix in which each coefficient on the diagonal is non zero\footnote{By definition a lower triangular matrix $A$ (possibly not square) is  such that the coefficient $a_{ij}$ is $0$ whenever $j>i$, and its diagonal is the sequence $(a_{ii})$.}  and $ 0 $ is $k \times (m-l) $ zero matrix. Thus $ \phi $ is equivalent to the formula:
$$\exists   \overline{y} \mesp \overline{y}. PAQ = \overline{x}. BQ. $$
So, we can replace $ A $  by $ A '$ and $ B $ by $ BQ$ and  assume that $ A $ is of the form
$ (A_1, 0) $, as described above. It follows that $ \phi $ is of the form $ \phi_1 (\overline{x}) \wedge \phi_2 (\overline{x}) $, where $\phi_1$  can be written as
  $$ \exists \overline{y} \mesp \overline{y}. A_1 = \left (\sum_{i = 1}^n x_i.b_{i, 1}, \dots, \sum_{i = 1}^n x_i.b_{l, l} \right)$$
  and $ \phi_2 $ as
$$\left (\sum_{i = 1}^n x_i.b_{i, l +1}, \dots, \sum_{i = 1}^n x_i.b_{i, m} \right) = (0, \dots, 0). $$
Since the diagonal coefficients of $ A_1 $ are all non-zero, any divisible $R$-module  satisfy $\forall \overline{x} \, \ \phi_1 (\overline{x}) $. Hence
$ \phi $ is equivalent to $\phi_2$. In particular, if $ x $ is a single variable, we have: $$ \phi (x) \iffl \bigwedge_{j = 1}^m x.b_{1, j} = 0 \iffl x.\text{gcd} (b_{1, 1}, \dots, b_{1, m}) = 0 .$$
\end{proof}
\end{proposition}
\begin{corr}\label{thm=thM_tor}
Completions of the theory of divisible non zero  $R$-modules admit elimination of quantifiers. In addition if $M$ is a divisible $R$-module such that $M_{tor} \neq 0$ then  $M_{tor}$ is an elementary substructure of $M$.
\end{corr}
\begin{proof}
The first assertion follows directly from the above proposition and the fact that modulo a complete theory of $R$-modules, any formula is equivalent to a boolean combination of p.p. formulas.

For the second assertion, first observe that $M_{tor}$ is a pure submodule of $M$ since it  is divisible (lemma \ref{Mtordivisible}). Thus, by \cite{sabbagh},  it is enough by  to  show that $M_{tor}$ and $M$ are elementary equivalent. To obtain this we will show that
for p.p. formulas $\phi(x)$ and $\psi(x)$  with $\phi(M) \subseteq \psi(M)$,
we have $\psi(M)/\phi(M) = \psi(M_{tor})/\phi(M_{tor})$. By the above proposition $\phi(x)$ and $\psi(x)$ are respectively equivalent to
$x.r=0$ and $x.q=0$ for some $r,q \in R$. If $r$ and $q$ are both non zeros, then there is nothing to show. If $r=0$ then $\phi(M)=M=\psi(M)$    necessarily. Hence it is enough to consider the non trivial case where $r \neq 0$ \& $q=0$. But $M_{tor}/\phi(M_{tor})$ is infinite in this case. In fact, if not, then choosing a set of representatives $x_1, \dots, x_n$ of $M$ modulo $\phi(M)$ and non zero polynomials $r_1,\dots ,r_n$ such that $x_i.r_i=0$, we have $M_{tor}.\lcm(r,r_1,\dots,r_n) = 0$.  This is impossible by Remark \ref{rmq:divisiblemodulesareinfinite}.
\end{proof}

To prove Theorem \ref{completions} we will study relative quotients of p.p. subgroups of a  divisible $R$-module $M$. They are of the form  $ann_M(r)/ ann_M(s)$ by  elimination of quantifiers.

Set $K_0:=\{x \in K \tq \vfi(x)=x\}$. Remark that every annihilator
set $ann_M(r)$ is a $K_0$-vector space. Until the end of this section, the symbol ``$\simeq$'' denotes isomorphism
between  $K_0$-vector spaces. The following lemmas establish a
generalization of lemmas 2.9 and 2.10 from {\cite{HP}} where it is assumed that $\varphi$ is onto, which means $R$ is both right and left euclidean.

\begin{lem} \label{TRdannpremier}
 Let $M$ be a divisible $R$-module. For all $0 \neq q = q_1 \dots q_k \in R $, $$\absolue{ann_M (q)} = \prod_{i = 1}^k \absolue{ann_M (q_i)}, $$
where $\absolue{ann_M(q_i)} $ is the cardinality of $ann_M(q_i)$.
\begin{proof}
If $q=sr$ then multiplication by $s$ induces an one to one $K_0$-linear transformation
$$ ann_M (q) / ann_M (s)  \longrightarrow ann_M (r)$$ which is in fact also onto by divisibility of $M$.
The result  clearly follows by induction on $k$.
\end{proof}
\end{lem}

\begin{lem}\label{annq_sur_ann_s} Let $T_0$ be a complete theory of non zero divisible $R$-modules. For all $s, q \in R \setminus
\{0\}$ such that $T_0 \models \forall x \, (x.s=0 \rightarrow x.q=0)$, there
exists $r \in R$ such that $$\left( ann_N (q) / ann_N (s) \right) \simeq
ann_N (r)$$ for all $N \models T_0$.

\begin{proof} Let $s$ and $q$ be as in the statement of the lemma. If $s$ divides $q$ then the result follows by the proof of the above lemma. Also if $ann_N (s)=0$ the assertion is trivial. Hence we assume that $ann_N (s)\neq 0$ for a model (hence for all models) $N$ of $T_0$. If $s$ is irreducible, then
$s$ is (up to a unit) the minimal polynomial of a non zero root of $q$ hence divides $q$. Now, consider the general case. We proceed by induction:
Suppose that for all non zero polynomials  $h,g$ such that
$ann_N (h) \supseteq ann_N (g) \neq 0$ and $g$ can be written as a product of $n-1$ irreducible polynomials, there exist $h',g',r'$ such that $g'=h'r'$, and
$$ann_N (g) / ann_N (h) \simeq
ann_N (g') / ann_N(h') \simeq ann_N(r')$$ for all $N \models T_0$.
  Consider the case $s=s_1\ldots s_n$ and set $s' = s_2 \ldots s_n$. Since $ann_N (s) \neq 0$, $s_1$
or $s'$ has a non-zero root in $N$.\\
{\bf Case 1 :} $ann_N (s_1) \neq 0$.  First, $s_1$ divides $q$ since it is irreducible and hence is -up to a unit- the minimal polynomial of any non zero $x \in
ann_N(s_1)\subseteq ann_N(q)$. Thus we have $q = s_1r$ and, by the fact
that $N$ is divisible, $ann_N (s') \subseteq ann_N (r)$ for some $r \in R$. If
$ann_N (s') = 0$ then we have: $ann_N (q) / ann_N (s) \simeq
ann_N (q) / ann_N (s_1) \simeq ann_N
(r)$. Otherwise, by induction hypothesis there are $s'_t$ and $r'$ such
that $s'_t $ divides $r'$ and
$$ann_N (r') / ann_N (s'_t) \simeq
   ann_N (r) / ann_N (s') .$$
Now the map ``$x \mapsto x.s_1$ mod $ann_N (s')$" establishes a
morphism of $K_0$-vector spaces from $ann_N (q)$ onto $ann_N (r) /
ann_N (s')$ with kernel $ann_N (s)$.\\
{\bf Case 2 :} $ann_N (s_1) = 0$. Then, $ann_N (s') \neq 0$,
$$ann_N (s) .s_1 =
ann_N (s') \subseteq ann_N (q) .s_1$$ and the
action of $s_1$ induces an isomorphism $K_0$-vector spaces: $$ann_N (q) \to
ann_N (q).s_1 \quad .$$ Therefore $$ann_N (q) /
ann_N (s) \simeq ann_N (q).s_1 /
ann_N (s').$$ By lemma  \ref{pp2atom} the p.p. formula $\exists y\mesp  y.s_1 = x
\wedge y.q = 0$, defining $ann_N (q) .s_1$, is equivalent to a formula of the
form: $x.q_1 = 0$ for some $q_1 \in R$. By applying the induction
hypothesis to $(s', q_1)$ we get $r$ as required.

It is clear from the proof that $r$ does only depend on $T_0$, $s$ and $q$.
\end{proof}
\end{lem}

\noindent
{\it Proof of  Theorem \ref{completions}}. The theorem follows  from \ref{ppindexs},  \ref{thm=thM_tor} and \ref{annq_sur_ann_s}. \qed

\medskip
We get by quantifier elimination the following consequence of Theorem \ref{completions}.
\begin{corr} \label{isomorphismetors}
Suppose $\eta_M=\eta_N$. Let $f$ be a
partial isomorphism between  submodules $A \subseteq M $ and $B \subseteq N$ where $N$ is $|M|^{+}$-saturated. Then the restriction of $f $ to $A \cap M_{tor} $ admits an extension to an  embedding of $M_{tor} $ in $N_{tor}$.
\end{corr}

\medskip

In the rest of  this section we will describe the {\it building blocks} of divisible $R$-modules and improve Theorem \ref{completions}.

\begin{nota} Let $T_d$ denote  the theory of divisible $R$-modules.
\end{nota}

\begin{defn}
A module $M $ is said to be indecomposable if there are no submodules $N_1, N_2 $, both non zeros, such that
$M = N_1 \oplus N_2$.
\end{defn}

\begin{lem} \label{indecomptor}
Let $N_1 $ and $N_2 $ be indecomposable models of $T_{d} $ and $q \in R $ an irreducible polynomial such that $ann_{N_i} (q ) $ is non-trivial for $i = 1,2$. Then $N_1$ and
$N_2 $ are isomorphic.
\begin{proof}
Let $a \in ann_{N_1} (q) $ and $b \in ann_{N_2} (q) $. Then $N_1 $ is a divisible closure of $a.R$ and $N_2 $ is a divisible closure of $b.R $. So it suffices to show that $a.R$ is isomorphic to $b.R$. In fact, the application  $a.r \mapsto b.r$ is an isomorphism of $R $-module since for all $r \in R \setminus \{0 \} $, $a.r = $0 if and only if $q $ divides $r $, if and only if $b.r = 0 $.
\end{proof}
\end{lem}

\begin{lem} \label{torsionpremier}
Let $N $  be an indecomposable model of $T_d$,   $r \in R $ a irreducible polynomial such that $ann_N (r) \neq 0 $. Then, for all irreducible $q \in R$ we have: $ann_N (q) \neq 0 $ if and only if, there exist $\lambda, \mu \in K \setminus \{0 \} $ such that $q = \lambda r \mu $; in which case $ann_N (r) \simeq ann_N (q) $.
\begin{proof}
 Let $0 \neq a \in ann_N (r) $ and $0 \neq b \in ann_N (q) $. Then $N $ is a divisible closure of both $a.R$ and $b.R$. Let $r_0$ be of minimal degree such that $0 \neq a.r_0 \in b.R$. So $r_0$ divides $r $. Since $r $ is irreducible, $a.r_0 \neq 0 $ implies $r_0 \in K $.
Therefore $a \in b.R $, i.e. $a = bs$ for some $s \in R \setminus \{0 \} $; take such an $s$ of minimal degree. Then $s$ divides $q $. Since $q $ is irreducible and $bs\neq 0$, $s \in K $. Moreover, the fact that  $b.sr = 0 $ implies  $q $ divides $sr$.  Since $s \in K $, $sr $ is also irreducible. Therefore $\deg (q) = \deg (r) $ and $sr = q \nu $ with $s, \nu \in K $.

Conversely, suppose $\lambda, \mu $  are as  in the statement. Then  the map $x \mapsto x. \lambda $ establishes an isomorphism of $K_0 $-vector space between $ann_N (q) $ and $ann_N (r) $.
\end{proof}
\end{lem}

\begin{defn} Let $r, q \in R $ be irreducible polynomials. We say that $r $ and $q$ are $K $-conjugate if there exist $\lambda, \mu \in K \setminus \{0 \} $ such that $q = \lambda r \mu $.
\end{defn}

\begin{remark}
$K $-conjugation is an
equivalence relation.
\end{remark}

\begin{nota}For the rest of this section we fix a set  $\mathcal{P}$  of representatives of irreducible polynomials modulo $K$-conjugation. Note that
the map by \ref{TRdannpremier}, the map $\eta_M$ is entirely determined by its restriction on $\mathcal{P}$ for any divisible $R$-module $M$.
\end{nota}

\begin{lem}
Let $r \in \mathcal{P} $ and $N$  be an  indecomposable model of $T_d$ containing a
non zero root of $r $. Then for all $q, s \in R \setminus \{0 \} $ such that
$ann_N (q) \supseteq ann_N (s)$, if
$\absolue{ann_N (q) / ann_N (s)} $ is finite, then it is equal
to $\absolue{ann_N (r)}^k $, for some $0 \neq k \in \N $.
\begin{proof}
By  Lemma \ref{annq_sur_ann_s} there exists $r'
\in R $ such that $$ann_N(q) / ann_N (s) \simeq ann_N
(r'). $$ Now,  $r'=r_0, \dots r_n$  with the $r_i$ irreducible. We have on
one hand $$\absolue{ann_N (r')} = \prod_{i \leq n} \absolue{ann_N (r_i)} $$
by  \ref{TRdannpremier},
and on the other hand, each $\absolue{ann_N (r_i)} \in
\{\absolue{ann_N (r)}, 1 \}$ by Lemma \ref{torsionpremier}.
\end{proof}
\end{lem}

\begin{proposition}
Any  divisible $R $-module is  the  direct sum of  indecomposable divisible submodules.
\begin{proof}
Any injective module over a noetherian ring  is the direct sum of  indecomposable submodules  (see \cite{lam} corollary 7.3) and  every direct summand of a divisible $R$-module is  divisible.
\end{proof}
\end{proposition}

\begin{nota}
Let $M \models T_{d} $ and $q \in R $ be an irreducible polynomial. We denote by
$M_q$ the sum of all divisible indecomposable submodules of $M_{tor} $ containing at least one non zero root of $q$.
\end{nota}
\begin{proposition}
One has $M_{tor} = \oplus_{q \in \mathcal{P}} M_{q} $.
\begin{proof}
It is enough to see that if $N \subseteq M_{tor} $  is a divisible indecomposable submodule of $M_{tor}$
 then $N $ is a divisible closure of some non zero element $a \in M$ annihilated by an irreducible polynomial $q $. First, since $N $ is indecomposable, it is a divisible closure of any of its non zero elements. Then, if $q_x $ is the minimal polynomial of a non zero $x \in N $, $q_x$ can be written as
$q_1 \dots q_k $, where the $q_i$ are irreducible. Hence $a:=x. (q_1 \dots q_{k-1})$ is annihilated by $q_k$.
\end{proof}
\end{proposition}

Hence we obtain the following improvement of  theorem \ref{completions}.

\begin{corr} \label{completionsdesRmoddiv}
Completions of $T_{d} \cup \{\exists x \mesp x \neq 0 \} $ are obtained by specifying for each $q \in \mathcal{P} $,
$\absolue{ann (q)} \in  \{\absolue{K_0}^n \tq n\in \mathbb{N}\} \cup \{\mathbb{\infty} \}$, and they admit elimination of quantifiers.
\end{corr}

\subsection{The case of $K[t; Frob]$-modules}
Let $ K $  be a field of characteristic $ p> 0 $ and $Frob$ be the Frobenius endomorphism of $K$. We say that $K$ is $ p $-closed if every polynomial of the form
$ X^{p^n} + X^{p^{n-1}}a_{n-1}+ \dots + X^pa_1+ Xa_0 + b \in K[X]$ has a root in $ K $. This amounts to say that $ K $ is divisible as a
$ K [t; Frob] $-module. We  have also  the following result which was first shown by Whaples (cf. \cite{whaples} ) and then a more elementary proof was provided by  Delon in her thesis (cf. \cite{delon}).
\begin{proposition}\label{PCLOS}
A field of characteristic $ p> 0 $ is $ p $-closed if and only if it has no extension of degree divisible by $ p $.
\end{proposition}

\begin{corr}
If $ L $ is a field of characteristic $ p> 0 $, $ p $-closed, then every algebraic extension of $ L $ is $ p $-closed.
\end{corr}

\begin{theorem}
 Let $ K \subseteq F $  be an extension of fields of characteristic $ p> 0 $. Then the torsion submodule  $F_{tor} $ of the $ K[t; Frob] $-module $ F $ is equal to the algebraic closure of $ K $ in $ F $. In addition, two $ p $-closed algebraic extensions  $ F_1 $ and $ F_2 $ of $ K $ are elementary equivalent as $ K[t; Frob] $-modules if and only if $ F_1 $ and $ F_2 $ are isomorphic as $K[t; Frob]$-modules.
\begin{proof}
Any algebraic  element over  $ K $ is a zero of an additive polynomial  by the fact that any polynomial divides an additive polynomial (see \cite{ore}).
This shows the first assertion.

Since $ F_1 $ and $ F_2 $ are $ p $-closed, by Corollary \ref{thm=thM_tor}, their theories as $ K[t; Frob] $-modules are given by the theories of their torsion submodules. Since $ F_1 $ and $ F_2 $ are algebraic over $ K $, $ F_1 = F_ {1_ {tor}} $ and $ F_2 = F_ {2_ {tor}} $. By  Proposition \ref{torsionfini},  they are isomorphic.
\end{proof}
\end{theorem}
\begin{corr}
Consider a field $K$ of characteristic $p>0$, p-closed and the following  tower of extensions:
$$ K \subseteq L \subsetneq M \subseteq K^{alg}.$$  Then $ L \not \equiv M $ as $ K [t; Frob] $-modules.
\end{corr}

\section{Valued Modules}
We are going to generalize the following situation:\\[2mm]
Let $(U,v)$ be a valued field  of equal characteristic $p>0$ with
$K \subseteq U $ a subfield, on which the valuation $v$ is trivial. We consider the valued  group $(U,+,v)$ together with the right action of the ring $R=K[t;Frob]$ of additive polynomials over $K$. Consider $\tau:v(U) \to v(U)$, $\gamma \mapsto p\gamma$. The valued module structure of $(U,v)$
is  the two-sorted structure
$$ \langle K,+,(.r)_{r\in R}, v(U),\leq, \tau, v, \infty \rangle$$
where $\langle U,+, (.r)_{r\in R} \rangle$ is the (right)-$R$-module
structure on  $(U,+) $ and $\langle v(U),\leq, \tau,\infty \rangle$ is the ordered value
set structure  of $v(U)$ equipped with $\tau$.

Since $v$ is trivial on $K$, for all $x \in U$ and $a \in U$, we have $v(x^{p^m}a)=
v(x^{p^m})=p^mv(x)$ and if $v(x)\neq 0$ then
we have  $v(x^{p^m})\neq v(x^{p^n})$   whenever $n\neq m$. Hence by ultrametric inequality, for an
additive polynomial $P(X):=\sum_{i \in I} X^{p^i}a_i$ $(a_i \neq 0)$ over $K$, for all $x$ with $v(x)\neq 0$ we have

$$v(P(x))=\min_{i\in I}\{p^iv(x)\}=
\begin{cases}\label{tauregular} \tau^{\min I}(v(x)) \mesp \mbox{if} \mesp v(x)>0 \\
\tau^{\max I}(v(x)) \mesp  \mbox{if} \mesp v(x)<0 \end{cases}.$$

\subsection{$\tau$-chains}
Let $L_{V_0}$ be the language $\{\leq, \tau, \infty\}$.
\begin{definition}  A $\tau$-chain is a $L_{V_0}$-structure $\Delta$ satisfying the following  axioms:
\begin{enumerate}[1.]
\item $\Delta$ is linearly ordered by $\leq$ with $\infty$ being its maximum,
\item $\tau$ is strictly increasing on $\Delta$ and $\tau(\infty)=\infty$,
\item $ \forall \gamma \neq\infty \, \forall \delta \mesp \left(\tau(\gamma)\leq\gamma \wedge \delta<\gamma\right) \rightarrow \tau(\delta)<\delta\label{tauvaluationnel0}$.
\end{enumerate}
\end{definition}
\begin{example}
For all integer $n\geq 1$ the structure $\langle \Z\cup \{\infty\}, \leq, x\mapsto nx\rangle$ satisfies the above axioms.
\end{example}

Remark that  axiom \ref{tauvaluationnel0} imply its dual:
\begin{equation} \forall \gamma \, \forall \delta \neq\infty \mesp \left(\tau(\gamma)\geq\gamma \wedge \gamma<\delta\right) \rightarrow \tau(\delta)>\delta.\label{dualtauvaluationnel0}\end{equation}

We denote:
\begin{itemize}
\item $\Delta_{+}:=\{\gamma \in \Delta \tq \tau(\gamma)>\gamma\}\cup\{\infty\}$   \item $\Delta_{-}:= \{\gamma \in \Delta \tq \tau(\gamma)<\gamma\}$

\end{itemize}
Thus,  $\Delta_{+}$ is a final segment  and $\Delta_{-}$ is an
initial segment of $\Delta$. In addition,  by Axiom  $\ref{tauvaluationnel0}$  there exists at most
one fixed point of  $\tau$ other than $\infty$: if it exists it is the unique point $\theta$ of $\Delta$ such that $\Delta_{-} < \theta < \Delta_{+}$. We define:
\begin{itemize}
\item the predicate $\delta < \theta$,    saying $\delta \in \Delta_{-}$  and,
\item the predicate $\delta>\theta$,  saying $\delta \in \Delta_{+}.$
\end{itemize}
Similarly the expressions $\delta \leq  \theta$ or $\delta \geq \theta$, does not mean that the cut $(\Delta_{-},\Delta_{+})$
is realized in $\Delta$.

\subsection{($K$-trivially) valued modules}

Recall that a valued abelian group is a structure given by
\begin{itemize}
\item an abelian group $M$,
\item a linear order $\Delta$ with a maximum element $\infty$,
\item a surjective map $v:M\to \Delta$ such that
\begin{itemize}
\item for all $x\in M$, $v(x)=\infty$ if and only if $x=0$
\item for  all $x,y \in M$, $v(x\pm y)\geq \min\{v(x),v(y)\}.$
\end{itemize}
\end{itemize}
These axioms imply that for all
$x,y \in M$,  $v(x\pm y)=\min\{v(x),v(y)\}$, whenever $v(x)\neq v(y)$, and
$v(x)=v(-x)$.

In a valued group $(G,v)$, the valuation  $v$ induces a topology, a basis of which
is given by the  {\it open balls}: these are subsets of $G$ of the form
$\{x\in G \tq v(x-a)>\gamma\}$, where $a\in G$ (center) and $\gamma \in v(G)$ (radius). We define
{\it closed balls} as usual  by changing $>$ to $\geq$.

\begin{nota}
In $(G,v)$, for all $\gamma \in v(G)$ we denote by $G_{\geq \gamma}$ (respectively $G_{>\gamma}$)  the closed ball (respectively the open ball) centered at $0$ with radius  $\gamma$.
\end{nota}

For the rest of this article we fix a field $K$  and we let $R:=K[t;\vfi]$
with $\vfi \in \text{End}(K,+,\times,1,0)$.

\begin{defn}\label{modktrivialementvalue}
Let $(M,v)$ be a valued abelian group with $v:M\to \Delta$. A $K$-trivially  valued module  structure on $(M,v)$ is given by  a right $R$-module structure on $M$
such that
\begin{enumerate}[1.]\label{axiomesmodvaluestriviales}
\item the function
$x \mapsto x.t$ is injective,
\item $\Delta$ is a $\tau$-chain,
\item $\forall x\in M$, $v(x.\lambda)=v(x)$ for all $\lambda \in K^{\times}$ \label{actiontrivideK}
\item $\forall x\in M$, $v(x.t)=\tau(v(x))$. \label{actiondeTau}
\end{enumerate}
\end{defn}

The name $K$-trivial comes from  Axiom \ref{actiontrivideK}. In \cite{gonenc} we consider also different (``non trivial")
actions of $K$ but in the present paper we will only deal with $K$-trivially valued modules. Therefore we permit our self
to omit the expression $K$-trivial and say only {\bf valued module} until the end of this paper.
\begin{remark}\label{remarkM_machantestconvexe}
When we deal with a valued module $(M,v)$,  we  write
$M_{>\theta}$ or $M_{\geq \theta}$ independently of the existence of $\theta$.
\end{remark}

%We will first  consider valued modules as two-sorted first order structures
%in the language $L_0:=L_{Mod}(R)\cup L_{V_{0}}\cup\{v\}$.

\begin{remark}
One has
$v(x) > \theta$ if and only if $v(x.t) > v(x)$, and $v(x)<\theta$ if and only if  $v(x.t)<v(x)$.
\end{remark}

%\begin{defn}
%Let $r \in R\setminus\{0\}$. We call $r$ separable if $r$  has constant term  $\neq 0$.
%\end{defn}
%Note that we understand the constant term in the sense of $R$. An additive polynomial $P\in F[X]$ over a field $F$ of characteristic $p>0$ has always its constant term equal to $0$ in the sens of $F[X]$. It is separable if and only if the coefficient of $X$ is non zero, equivalently
%if and only if, seen as an element of  $FnFrob]$, $P$ has its constant term equal to $0$.

\begin{defn}
A polynomial $r\in R$ such that $t$ does not divide $r$ is called separable.
\end{defn}
\begin{remark} Any irreducible $r\neq t$ is separable. An $R$-module is divisible if and only if it is divisible by separable polynomials and by $t$, in which case we will say $t$-divisible for short.
\end{remark}
\begin{nota}
Using Equation \ref{eqirr} we observe that any  non zero polynomial $q \in R$ can be written in the form $q=t^ns$ where $s$ separable. We  set $$\deg_{is}(q):=n .$$
\end{nota}

\begin{lem}\label{ConsT}
Let $(M,v)$ be a valued module and  $r=\sum_{i\in I}t^ia_i \in R\setminus \{0\}$. Then, for all $x \in M$, we have  $v(x.r)\geq \min_{i \in I} \{v(x.t^ia_i)\}=\min_{i\in I}\{\tau^i(v(x))\}$. In addition for all $x \in M$,
\begin{enumerate}[$1.$]
\item  if $v(x)>\theta$ then $v(x.r)=\tau^{\deg_{is}(r)}(v(x))=\min_{i\in I}\{\tau^i(v(x))\}$,
\item if  $v(x)<\theta$ then $v(x.r)=\tau^{\deg (r)}(v(x))=\min_{i\in I}\{\tau^i(v(x))\}$,
\item if $v(x.r)>min_{i\in I}\{\tau^i(v(x))\}$ and $x\neq 0$ then  $v(x)=\theta$.
\end{enumerate}
\begin{proof}
Let $x \in M\setminus\{0\}$.  By the valued abelian group structure of $(M,v)$  we have:  $v(\sum x.t^ia_i)\geq v(x.t^ia_i)$ for all $i \in I$. Then we have $v(x.t^ia_i)=v(x.t^i)=\tau^i(v(x))$ by Axioms \ref{actiontrivideK} and \ref{actiondeTau} of Definition \ref{axiomesmodvaluestriviales}.
\begin{enumerate}[1.]
\item If $v(x)>\theta$ then  $v(x.t)>v(x)$. Now, $$v(x.t^ia_i)=v(x.t^i) > v(x.t^j)=v(x.t^ja_j)$$ whenever $i > j$. Hence, $v(x.r)=v(\sum x.t^ia_i)= \min_i\{v(x.t^i)\}=\tau^k(v(x))$, where $k=\deg_{is}(r)$.
\item  In this case we have   $v(x.t)<v(x)$  and hence   $v(x.t^ia_i) < v(x.t^ja_j)$ whenever $i > j$. Thus, $v(x.r)=v(\sum x.t^ia_i)= \min_i\{v(x.t^i\}=\tau^n(v(x))$ where $n=\deg(r)$.
\item By points $1.$ and $2.$ above, $x$ can only have valuation $\theta$.
\end{enumerate}\vspace{-0.465cm}\end{proof}
\end{lem}

\begin{corr}
The subsets $M_{>\theta}$ and $M_{\geq \theta}$ are   $R$-submodules of $M$.
\end{corr}

\begin{defn}\label{defconvexe}
Let $(G,v)$ be an abelian valued group. A subgroup of $G$ is called convex  if
it is the inverse image under  $v$ of a non empty final segment  of $v(G)$.
\end{defn}

\begin{lem}\label{M_>theta-convexe}
A (closed or open) ball centered at $0$ is a convex subgroup. If $(M,v)$ is a valued module then $M_{\geq \theta}$ and  $M_{>\theta}$ are convex subgroups of $M$.
\begin{proof}
A ball centered at $0$ with radius $\gamma$ is the inverse image under $v$ of the final segment $(\gamma, \infty]$ or  $[\gamma,\infty]$. Hence $M_{\geq \theta}$ and $M_{>\theta}$ are convex subgroups  if $\theta \in v(M)$. Otherwise,  \ref{remarkM_machantestconvexe},  $M_{\geq \theta}=M_{>\theta}=v^{-1}(\Delta_{+})$.
\end{proof}
\end{lem}
\begin{lem}\label{valuationconvexe}
If $H$ is a  convex subgroup of an abelian valued group  $(G,v)$ then  the quotient
$G/H$ is valued by $v_H$, defined by: $$ v_H(g + H):= \begin{cases}v(g) \mesp \text{if} \;  g\notin H \\  \infty \mesp \text{otherwise}\end{cases}.$$
\begin{proof}
It is sufficient to verify that $v_H$ is well defined. Let $I_H$ be the final segment of  $v(G)$ such that $H=v^{-1}(I_H)$. If  $g\notin H$ and $g'\notin H$ are two representatives  of the same class, then $v(g) \notin I_H$, $v(g') \notin I_H$ and  $g-g' \in H$. Thus  $v(g-g') \in I_H$ and   $v(g-g')>v(g)$ necessarily. It follows that $v(g)=v(g')$.
\end{proof}
\end{lem}
\begin{corr}\label{cor:quotientmoduleisvalued}
Let $(M,v)$ be a valued module. Then $M/M_{\geq \theta}$ is canonically equipped   with a valued module structure.
\end{corr}

\begin{defn}
A valued module $(M,v)$ will be called regular if,  for all $x\in M$ and non zero
$r=\sum_{i=0}^nt^ia_i \in R$, we have  $$v(x.r)=\min_{i}\{v(x.t^i)\}=\min_{i}\{\tau^i(v(x))\}.$$
If the above equality holds for a pair $(x,r)$ we say that $x$ is regular for $r$ (otherwise irregular for $r$) and $x$ is said to be regular
if it is regular for all $r$ (otherwise irregular).\end{defn}
\begin{remark}\label{irregelts}
 By \ref{ConsT} $x$ is irregular if and only if $v(x)=\theta$ and $v(x.r) \in M_{>\theta}$ for some non zero $r$.
\end{remark}
\begin{remark} Regular modules are necessarily torsion free.
\end{remark}

The following lemma follows directly by  Corollary \ref{cor:quotientmoduleisvalued} and by  Lemma \ref{ConsT}.
\begin{lem}\label{ConsTcorr} Let  $(M,v)$ be a valued module. The submodule
$(M_{>\theta},v)$ and the quotient
$\left(M/M_{\geq \theta}, v_{M_{\geq \theta}}\right)$ are regular valued modules.
\end{lem}

\begin{defn}
We say that  a sequence of submodules $(A_i)_{i \in I}$ of $M$ is  valuation independent  if, for any sequence $(x_i)_{i \in I}$ with
$x_i \in A_i$ and for all finite $J$ with $J\subseteq I$, we have
$$v(\sum_{i \in J} x_i)
=\min\{v(x_i)  \tq i \in J\}.$$
\end{defn}
\begin{remark}
If  $(A_i)_{i \in I}$ are torsion free and valuation independent
then  any sequence $(x_i)_{i \in I}$ with $x_i \in A_i$, is
$R$-linearly independent.
\end{remark}
\begin{fact}\label{faitreg} Let $(M,v)$ be a valued module,  $A \subseteq M$ be a regular submodule and $B$ a submodule of  $M_{tor}$. Then
 $A$ and $B$  are valuation independent.
 \begin{proof}
Let $x \in M_{tor}$, and $a\in A$. Without loss of generality we may assume both $x$ and $a$  are non zero. If $v(a-x)>v(x)=v(a)$ then $v(a)=\theta$ and since $a$ is regular $v(a.r)=\theta$. But now, for some non zero $r$ such that $x.r=0$ we have $\theta=v(a.r)<v((a-x).r)$, a contradiction.
\end{proof}
\end{fact}

%\begin{proof}
%The fact that  $\left(M/M_{\geq \theta}, v_{M_{\geq \theta}}\right)$ has an induced valued module  structure  follows from  lemma \ref{valuationconvexe} and from the fact: if $x \notin M_{\geq \theta}$ then $v(x.t)<v(x)$. Then these are regular valued modules by lemma \ref{ConsT}.\end{proof}

\subsection{Henselian valued modules}
\begin{defn}\label{defhensel}
A valued module  is said to be henselian if it satisfies the following  axiom scheme:
\begin{enumerate}[H]
\item : $\forall x \mesp v(x)>\theta \rightarrow \exists y \mesp x=y.r\wedge v(y)>\theta$\\
for all separable $r \in R$.\label{henseltrivi}
\end{enumerate}
We denote by $T_h$ the theory of  henselian valued modules.
\end{defn}
%Models of $T_h$  that we want to study  include indeed the  field
%$(\F_p(X))^{alg}$,  with its $X$-adic valuation  when we consider  it as module over $\F_p^{alg}[t;Frob]$.
%Moreover:
\begin{lem}\label{henscommecorsp_imp_hens_commemodule}
Let $(K \subseteq U,v)$ be an extension of valued fields  of characteristic $p>0$, where $v$ is trivial on $K$. Then $(U,v)$   canonically inherits a valued $K[t;Frob]$-module structure. In addition, if $v$ is  henselian on the field  $U$, then $(U,v)$ is  a henselian valued module. Moreover, if  $U$ is perfect, then the maximal ideal  $\mathcal{M}_U$ associated to $v$ is a divisible $K[t;Frob]$-module.
\begin{proof}
As usual we interpret $x.t$ as  $x^p$, $\tau$ as the map $\tau:v(U) \to v(U)$,
$\gamma \mapsto p\gamma$, and $\theta$ as $0 \in v(U)$ which makes $(U,v)$ a valued module. Suppose now that $(U,v)$ is henselian as a valued field. Let $q=t^na_n + \dots + a_0$ be separable and
$y \in U$ such that $v(y)> 0$. Let $Q$ be the additive polynomial associated to $q$: i.e.
$Q(X)=a_nX^{p^{n}} + \dots + a_0X$. Since $a_n,\dots, a_0 \in K$, they have common valuation $0$, and since $q$ is separable, $a_0 \neq 0$. Then, by setting $F(X):=Q(X) - y$, we have $F'(0)=a_0$ and $F(0)=-y$. Since $v$ is henselian, there exists $z \in U$ of valuation $>0$ such that $F(z)=0$. Thus $Q(z)=y$ and $v(z)>0$. In other words,  in the language of valued modules, we have  $z.q=y$ and $v(z)>\theta$.

Finally, remark that $U_{> \theta}=\mathcal{M}_U$ and if $U$ is perfect, then   $U_{> \theta}$ is   divisible by $t$. Since by the above paragraph it is also divisible by separable polynomials, it is divisible.
\end{proof}
\end{lem}

\begin{lem}\label{tfh_gives_reg}
Any torsion free henselian valued module is regular.
\end{lem}
\begin{proof}
Suppose for a contradiction that we have a henselian torsion free module
$(M,v)$, with an element $x \in M$ irregular. Then for some non zero $r\in R$, $v(x.r)>\theta$ and $v(x)=\theta$. Write $r=t^ns$ with $s$ separable. By henselianity there exists $y \in M$, of valuation
$>\theta$ such that $y.s=x.r$. Then $y-x.t^n $ is non zero but annihilated by
$s$. Contradiction.
\end{proof}

%\begin{defn}
%Immeat extensions.
%\end{defn}
%
%\begin{theorem}
%Any valued module has
%\end{theorem}
It is a trivial fact that algebraically closed valued fields are henselian. We could look to the notion of divisibility  for valued modules as an analogue of the notion of algebraic closeness for valued fields. But there exist divisible valued modules which are not henselian as it is showed in the following proposition.
%
%\begin{lem}
%Any torsion free henselian valued module is regular.
%\end{lem}
%\begin{proof}
%Suppose for a contradiction that we have a henselian torsion free module
%$(M,v)$, with an element $x \in M$ witnessing non regularity of $M$. Then by \ref{ConsT} $v(x)=\theta$ and for some non zero $r\in R$, $v(x.r)>\theta$. Write $r=t^ns$ with $s$ separable. By henselianty there exists $y \in M$, of valuation
%$>\theta$ such that $y.s=x.r$. Then $y-x.t^n \neq 0$ but annihilated by
%$s$. Contradiction.
%\end{proof}

\begin{proposition}\label{contre-exemple}
There exist non henselian divisible valued modules.
\begin{proof}
Let $U$  be an algebraically closed field of characteristic  $p>0$, $v$ a non trivial valuation on $U$ and $K$ a trivially valued subfield. We consider $U$ with its  $K[t;Frob]$  valued module structure.

Let $y \in U$, of valuation $>0$. Then,  by  \ref{ConsT}, $y$ is not a  torsion element. Since $1$ is annihilated  by $(t-1)$, it is a torsion element (indeed $1.t=1^p=1$). Consider, the submodule $A:=(1+y).R$. It is torsion free. Set  $x=(1+y).(t-1)\in A$. Since $1.(t-1)=0$, we have also $x=y.(t-1)$. On the other hand, since $U$ is algebraically closed, it is divisible as an $R$-module. Hence $U$ contains a  divisible closure of  $A$. Let the module $B$ be  such a closure. By Lemma \ref{modinjective} (point 4) $B$ is a torsion free module. Thus $1 \notin B$, hence   $y \notin B$: $B$ can not be henselian since $y$ is the  unique element of $U$ of valuation $>0$ such that $y.(t-1)=x$.
\end{proof}
\end{proposition}

\begin{corr}\label{rem-contre-ex} Take $x$ as above. Then $x.R$ has two divisible closures non elementary equivalent as valued modules.
\begin{proof} Since $v(x)>\theta$ and $U_{>\theta}$ is divisible, $x.R$ has a divisible closure inside $U_{>\theta}$, which is necessarily henselian. This can not be  elementary equivalent to $B$, since $B$ is not henselian.
\end{proof}
\end{corr}
\subsection{Henselian divisible valued modules}
The  results of this subsection contain the essential  information that will be used to establish an  Ax-Kochen and Ershov principle.
\begin{lem}\label{henseldivisible}
If  $(M,v)$ is a henselian divisible valued module,  $M_{\geq \theta}$ is divisible and $M_{>\theta}$ is divisible  torsion free. Hence $M_{\geq \theta}$ is a direct summand in $M$, and  $M_{>\theta}$ is a direct summand  in  $M_{\geq \theta}$.
\begin{proof}
Let $r \in R\setminus\{0\}$, $x \in M_{\geq \theta}$  and  $y \in M$ such that $y.r=x$. Then by \ref{ConsT} $y\in M_{\geq \theta}$. Thus $M_{\geq \theta}$  is a divisible submodule of $M$.

Now we show that  $M_{> \theta}$ is divisible. Let $x$ be of valuation $>\theta$ and
$r \in R$, non zero.
Write  $r=t^n.r'$ with $r'$ separable, then there exists, by axiom $H$, $y \in M_{>\theta}$ such that $x=y.r'$. By $t$-divisibility of $M$, we have  $y=y'.t^n$ for some $y' \in M_{>\theta}$, hence $y'.r=x$. Moreover $M_{>\theta}$ is  torsion free since it is a regular submodule.
\end{proof}
\end{lem}

\begin{corr}\label{obsmodT}
Let $(M,v)$ be a henselian divisible valued module. If $x \in M$ is irregular then there exists a unique couple $(x_{tor},x_{>\theta})$, with $x_{tor} \in M_{tor}$ and $x_{>\theta} \in M_{>\theta}$ such that
$$x=x_{tor} + x_{>\theta}.$$ As a consequence,
\begin{enumerate}
\item for all $x \in M$ and  $q=\sum_{i\in I}t^ia_i\in R\setminus \{0\}$, there exists regular $y \in M$ such that $x=y.q$ and hence $v(x)=\tau^{k}(v(y))$, where $k=\deg_{is}(q)$ if $v(x)>\theta$, $k=\deg(q)$ if $v(x)<\theta,$ and $k=0$ if $v(x)=\theta$; in any case $$v(x)=\min_{i \in I}\{v(y.t^ia_i)\}=\min_{i \in I}\{\tau^i(v(y))\},$$
\item  $0\neq x\in M$  is regular for $r$ (resp. regular) if and only if for all $a \in ann_M(r)$ (resp. for all $a\in M_{tor}$), $v(x-a)=\min\{v(x),\theta\}$.
\end{enumerate}
\begin{proof}
Take an irregular  $x\in M$. Let  $r \in R\setminus \{0\}$ be such that $x.r \in M_{>\theta}$ (given by \ref{irregelts}). Since
$M_{>\theta}$ is divisible,  $x.r=y.r$ for some $y\in M_{>\theta}$ and $y$ is necessarily regular. Since $M_{tor}\cap M_{>0}=0$
 the couple $(x-y,y)$ is unique as required. Then it follows from \ref{ConsT} that $v(y.r)=\tau^{\deg_{is}(r)}(v(y))$.
The first consequence mentioned  follows then by \ref{ConsT} and the second from the proof of the first one.
\end{proof}
\end{corr}

\begin{theorem}\label{lemdecomp} Let $(M,v)$ be a   henselian divisible valued module.
\begin{enumerate}[1.]
\item $M_{tor}$ embeds  in
$M_{\geq \theta}/M_{>\theta}$ and  for any such  embedding the image of $M_{tor}$ is a  direct summand.
\item The $R$-modules $M_{\geq \theta}/(M_{tor} + M_{>\theta})$, $M_{>\theta}$ and $M/M_{\geq \theta}$ are torsion free and divisible.
\item $M$ can be written as a direct sum: $$M_{tor} \oplus M_{\theta} \oplus M_{>\theta} \oplus M_{-},$$ where
$M_{\theta}$ is isomorphic to  $M_{\geq \theta}/(M_{tor} + M_{>\theta})$ and $(M_{-},v)$ is isomorphic as a valued module to
$(M/M_{\geq \theta},v_{M_{\geq \theta}})$.  This decomposition   is valuation independent and each member of this decomposition, except $M_{tor}$, is a regular $D$-vector space.
\end{enumerate}
\begin{proof}
$1.$ Since all non-zero elements of $M_{tor}$ are of valuation  $\theta$, the canonical surjection ($M_{\geq \theta} \to M_{\geq \theta} / M_{>\theta}$) induces an embedding of  $M_{tor}$ in
$M_{\geq \theta}/M_{>\theta}$. Since $M_{tor}$ is divisible its image in $M_{\geq \theta} / M_{>\theta}$ is a direct summand.

\noindent
$2.$ The fact that  $M_{>\theta}$ is torsion free and divisible is given by \ref{henseldivisible}. The divisibility of  $M/M_{\geq \theta}$  is induced by the divisibility of $M_{\geq \theta}$, it is torsion free by \ref{ConsTcorr} since it is a regular valued module. Also $M_{\geq \theta}/(M_{tor} + M_{>\theta})$ is divisible by divisibility of $M$. We will now show that $M_{\geq \theta}/(M_{tor} + M_{>\theta})$ is torsion free.
Let $x \in M_{\geq \theta}$ such that  $x.r \in M_{tor} + M_{>\theta}$
for some $r\neq 0$. Since $M_{tor}$ is divisible, we have $x.r-z'.r=(x-z').r \in M_{>\theta}$ for some $z' \in M_{tor}$. If $x-z' \notin M_{>\theta}$ then $x-z'$ is irregular. In this case, by  Corollary \ref{obsmodT},  $x-z' \in M_{tor} + M_{>\theta}$ and hence $x \in M_{tor} + M_{>\theta}$.\\
$3.$ Since $M_{\geq \theta}$ and $M_>\theta$ are both divisible,
$M \simeq M_{\geq \theta} \oplus M/M_{\geq \theta}$ and  $$M_{\geq \theta} \simeq M_{> \theta}
\oplus M_{\geq \theta}/M_{>\theta}$$ and   we get
 $M_{\geq \theta}/M_{>\theta} \simeq M_{tor} \oplus \left(M_{\geq \theta}/(M_{tor} + M_{>\theta})\right)$.

Take a direct summand $M_{-}$ of $M_{\geq \theta}$ in $M$. If $x \in M_{-}\setminus \{0\}$ then $v(x)<\theta$ and by the definition of the quotient valuation $v_{M_{\geq \theta}}$, $(M_{-}, v)$ and $(M/ M_{\geq \theta},v_{M_{\geq \theta}})$  are isomorphic as valued modules.

\noindent
It remains to show that this decomposition is   valuation independent. For this it is enough to see that  $M_{tor}$ and $M_\theta$ are  valuation independent.  Suppose that  there exist $x\in M_{tor}$, $y  \in M_{\geq \theta} \setminus \left(M_{tor} + M_{>\theta}\right)$ with $v(x-y)> \theta$. Then, for some $r \in R\setminus\{0\}$  annihilating  $x$, we have $v((x-y).r)=v(y.r)> \theta$: this is impossible since  $M_{\geq \theta}/(M_{tor} + M_{>\theta})$ is torsion free.
\end{proof}
\end{theorem}

\begin{corr}\label{corrvalindep}
For all $i\in \{\theta, >\mathrel{\theta}, -\}$, for all $x,y \in M_i\setminus\{0\}$ the following are equivalent:
\begin{enumerate}[$\bullet$]

\item $v(x-y)>v(x)=v(y)$,
\item there is some $r  \in R\setminus\{0\}$ such that $v(x.r-y.r)>v(x.r)=v(y.r)$,
\item for all $r  \in R\setminus\{0\}$ we have $v(x.r-y.r)>v(x.r)=v(y.r)$.
\end{enumerate}
\begin{proof}
Follows by the fact that each $M_{i}$ is regular.
\end{proof}
\end{corr}

\begin{remark}
Let $(K \subseteq U,v)$ be an extension of fields of  characteristic $p>0$, where $U$ is algebraically closed and $|K|^{+}$-saturated  and $v$ is trivial on $K$. Then  every non zero element of the set $U_{\theta}$ realizes the generic type of the valuation ring of $U$, i.e.
the type $\{x \in \mathcal{O}_U \wedge x \notin B \tq B\subsetneq \mathcal{O}_U \; \text{is}\;  K\text{-definable}\}$.
Note that here definable means definable in the language
$\{0,1,+,-,\times, \mathcal{O}\}$ of valued fields. \end{remark}

\subsection{Embedding theorems and A-K,E principles}

Let $(M, v)$ be a valued module. If $A$ is a submodule of $M$, Corollary \ref{rem-contre-ex} shows that the various  divisible closures of $A$ in $M$, while being isomorphic as $R$ -modules, may not be   elementary equivalent as valued modules. This is the most important phenomenon to which we will pay attention.

\begin{defn}
Let $(M, v)$ be a divisible valued module and $A \subseteq M$ a submodule of $ M $. Then we define $$ \hat {A}: = \{y \in M \tq y.r \in A \mesp \text{for some} \mesp r \in R \setminus \{0 \} \}.$$
\end {defn}

\begin{lem} \label{uniquecloturedivisiblesurTor}
The submodule $\hat{A}$ is the unique  divisible closure of $A + M_{tor}$ inside $M$. In particular, for all $r \in R \setminus \{0 \}$ and all $x \in M$, if $x.r \in \hat {A}$ then
$ x \in \hat {A}$. Moreover, if  $(M, v)$  is henselian then $\hat{A}$  is henselian.
\begin{proof}
By construction, $\hat {A}$ is a  divisible submodule and all divisible submodules of $M$ containing $A$ and $M_{tor}$ contains $\hat{A}$. This gives the uniqueness. Also by the definition of $\hat{A}$, if $x \in \hat{A} $ and $ r \in R \setminus \{0 \}$, $ \hat{A} $ contain all $ y $ such that $ x = y.r$ since
$\hat{A} $ contains $ M_ {tor}$. Now, if $ (M, v)$ is henselian and $ x \in
\hat{A}$ of valuation $>\theta $, then $\hat{A}$ contains the unique element  $ y_{> \theta } \in M_ {> \theta} $ such that $y_{> \theta}.r = x$. This shows that $ \hat{A} $ is henselian.
  \end{proof}
\end{lem}
\begin{remark}\label{tauclosure} The set $v(\hat{A}) $ is an $L_{V_0}$-substructure of $v(M)$, it is the closure of
$ v(A) $ under
the function $\tau^{-1}$.
\end{remark}

\begin{remark}\label{predicatesRn} Consider a henselian divisible valued module $(M,v)$.
Then,
\begin{enumerate}
\item $M_{\geq \theta}/M_{>\theta}=0$ or infinite,
\item if $\gamma \neq \theta$ and $r$ is non zero with $\deg_{is}(r)=k$ and $\deg(r)=n$
then
\begin{enumerate}
\item if $\gamma>\theta$ the multiplication by $r$
$$.r: M_{\geq \gamma} \to M_{\geq \tau^k(\gamma)}$$
is a bijection inducing a bijection
$$ M_{\geq \gamma}/M_{> \gamma} \to M_{\geq \tau^k(\gamma)}/M_{> \tau^k(\gamma)},$$
\item if $\gamma<\theta$ the multiplication by $r$
$$.r: M_{\geq \gamma} \to M_{\geq \tau^n(\gamma)}$$
is a bijection inducing a bijection
$$ M_{\geq \gamma}/M_{> \gamma} \to M_{\geq \tau^n(\gamma)}/M_{> \tau^n(\gamma)}.$$
\end{enumerate}
\end{enumerate}
\end{remark}
\begin{proof}
If $\theta \in v(M)$ then $M_{\theta}$ or $M_{tor}$ is non empty, divisible and embedded into $M_{\geq \theta}$. The second assertion follows from divisibility and regularity of any element of valuation
$\neq \theta$.
\end{proof}

-- For each $n \in \N$, we add  a unary predicate $R_n$ into our language $ L_ {V_0}$ and let $L_V$ be this enrichment of $ L_ {V_0}$. We denote $L_0:=L_{Mod}(R) \cup L_{V_0}  \cup \{v \}$ and
$L:=L_{Mod}(R) \cup L_{V}  \cup \{v \}$. In every valued module  $(M, v)$ and for all $\gamma \in v(M) \setminus \{\infty \}$, $R_n (\gamma)$ will be interpreted
by the equivalence  $$R_n(\gamma) \Leftrightarrow \absolue{M_{ \geq \gamma} / M_{> \gamma}} \geq n.$$

\begin{proposition}\label{Achapeau}
Let $ (M, v) $ and $ (N, w)$   be  divisible henselian valued modules  such that
$\eta_M = \eta_N$ and $N$ is $|M|^{+}$-saturated. Let $(A,\Delta_1) \subseteq (M, v(M))$ and $(B,\Delta_2) \subseteq (N,v(N))$ be  $L$-substructures of $M$ and $N$ respectively, which are $L$-isomorphic via ${\bf f} = (f, f_{v})$.  Then there is an
$L$-embedding  ${\bf f}=(\hat{f},\hat{f}_v)$ of $(\hat {A},v(\hat{A})\cup \Delta_1)$ to $(N,v(N))$ extending ${\bf f}$, having range
$(\hat{B}, v(\hat{B})\cup\Delta_2)$.

\begin{proof}
By Remark \ref{tauclosure} and the fact that $\tau$ is strictly increasing, it is easy to see that $f_v$ extends uniquely to an $L_{V_0}$-embedding $$\hat{f}_v:v(\hat{A})\cup \Delta_1 \to w(\hat{B})\cup \Delta_2.$$

Now we will extend $f$ to $\hat{f}:\hat{A} \to N$ such that $(\hat{f}, \hat{f}_v)$ is an $L$-embedding. Take a decomposition $$\hat{A}= M_{tor}\oplus \hat{A}_ {\theta} \oplus
\hat{A}_{-} \oplus \hat{A}_{>\mathrel{\theta}}$$ as given by Theorem \ref{lemdecomp}.

Let $ I := \{\theta, >\mathrel{\theta}, {-} \} $ and
for $ i \in I $ denote by $ f_i $  the restriction of
$ f $ to  $ A \cap \hat {A_i} $.
By the definition of $ \hat {A}$ and  by the fact that
$ \hat {A} _i $ is torsion free, if $x \in \hat {A}_i $ then there is $ r \in R \setminus \{0 \} $ such that $ x.r \in A$ and $x.r  \neq  0$ whenever $x\neq 0$,  Therefore,
each  $ \hat {A} _i $ admits  a $ D $-vector space basis
consisting of elements of $ A $.
Then, $ f_i $ extends uniquely  to  an isomorphism of $D$-vector spaces
$\hat{f}_i: \hat{A}_i \to \hat{f}_i(\hat{A}_i) \subseteq N$.
Now  we let
$C:=\bigoplus_{i \in I} \hat{A}_i$, define
$\hat{f}_C: C \to N$ as
$$x=x_{\theta} + x_{-} + x_{>\theta} \mapsto \hat {f}_{\theta} (x_ {\theta}) +
\hat {f} _ {-} (x_ {-}) + \hat {f} _ {> \theta}(x_ {> \theta})$$ and
$$A_t:=\{x \in M_{tor} \tq \exists  c\in C, \mesp  x+c \in A\}.$$
Note that $A_t$ is a submodule of $M_{tor}$.
We will see that the restriction of $f$ to $A\cap M_{tor}$ admits a unique extension $h$ to $A_t$ such that the map
$$ x = x_{tor} + x_{\theta} + x_{-} + x_{> \theta}\mesp \mapsto \mesp h(x_ {tor}) + \hat{f}_C(x-x_{tor}) $$ defined on  $A_t \oplus C$ is  a monomorphism of $R$-modules extending $f$. Let $x \in
A_t \setminus A$ and $c \in C$ be such that $x + c \in A$. We have to check that
$$x \mapsto h(x):=f(x+c) - \hat{f}_C(c)$$ is well defined.
Let $d \in C$ be such that $x+d \in A$. Note that $d-c \in A$. Now we have
\begin{multline*}f(x+d) - \hat{f}_C(d) - \left(f(x+c) - \hat{f}_C(c)\right) \\
= f(d-c) - \hat{f}_C(d -c)= f(d-c) - f(d-c)=0.\end{multline*}  By the definition of $h$, $h$ is a morphism of $R$-modules.
Moreover $h$ is injective: If for some $c \in C$, $x + c \in A$ and $f(x+c)-\hat{f}_{C}(c)=0$ then $\hat{f}_C(c) \in B$ hence
$c \in A$. It follows that $x\in A$ and $f(x)=0$ hence  $x=0$.

Set $\hat{B}_i:=\hat{f}(\hat{A}_i)$. By construction
the $\hat{B}_i$ are valuation independent, in particular, their sum is direct. Define  $E:=\bigoplus_{i \in I} \hat{B}_i$.  We will now check that the range of $h$ is exactly
$$B_{t}:=\{y \in N_{tor} \tq \exists z \in E \mesp y+z \in B\}.$$
Let $y \in B_t$ and $z \in E$ be such that $y + z \in B$. Let $r$ be the minimal polynomial of
$y$. Then $z.r \in B$ since $y.r=0$ and $(y+z).r \in B$. Let $c'\in C$ be the image of $z$ under $\hat{f}_C^{-1}$ and $x \in A$ be the image of $y+z$ under $f^{-1}$. Then $x_0:=x-c'$ is annihilated by $r$ and  $h(x_0)=y$.

Since $\eta_M=\eta_N$ and $N$ is $|M|^{+}$-saturated, by \ref{isomorphismetors}  we can extend $h$ to an elementary $L_{Mod}(R)$-embedding $M_{tor} \to N$, again denoted by $h$, having its range equal to $N_{tor}$.  Hence we can define an $L_{Mod}(R)$-embedding $\hat{A} \to \hat{B}$, denoted by
$ \hat{f} $,  by setting
$$\hat{f}(a):= h(a_ {tor}) + \hat {f}_C(c)$$
 where $a = a_{tor} + a_{\theta} + a_{-} + a_{> \theta} \in \hat{A}$
with $c=a_{\theta} + a_{-} + a_{> \theta}$.

Note now that each $\hat{B}_i$ is regular. In fact, if $\hat{f}_i(x)=y\in \hat{B}_i$ is irregular then, there exists a non zero $y_0 \in N_{tor}$ and $y_{>\theta}\in N_{>\theta}$ such that  $y=y_0 + y_{>\theta}$. Let $r$ be the minimal polynomial of $y_0$, then $y.r=y_{>\theta}.r \in B$. Since $\hat{B}_{>\theta}$ is torsion free, $y_{>\theta}=y$. Contradiction.

Denote by $\hat{f}_{v,i}$ the restriction of $\hat{f}_v$ on $v(\hat{A}_i)$.  Since each $\hat{B}_i$ is regular the $(\hat{f}_i, \hat{f}_{v,i})_i$ are $L_0$-embeddings. Moreover since by Fact \ref{faitreg} the decomposition
$N_{tor}\oplus E$ is valuation independent, we have
$w(h(x) + \hat{f}_C(x))=\min\{w(h(x), w(\hat{f}_C(x))\}$ for all $x\in \hat{A}$. Hence  $(\hat{f},\hat{f}_v)$ is an $L_0$-embedding. Now  it is in fact an $L$-embedding since
$\mathbf{f}$ is an $L$-embedding and, by Remark \ref{tauclosure},   $v(\hat{A})$ is the closure of $v(A)$ by $\tau^{-1}$ and by Remark \ref{predicatesRn} each quotient $\hat{A}_{\geq \gamma}/\hat{A}_{>\gamma}$ is entirely determined by a quotient of the form $A_{\geq \tau^k(\gamma})/A_{>\tau^k(\gamma)}$.
\end{proof}
\end{proposition}

\begin{nota}
For $a \in M$, where $M$ is valued by $v$, and $\gamma \in v(M)$, we denote by
$M_\gamma(a)$ the quotient of the closed ball $\{u \in M \tq v(u-a)\geq \gamma\}$ given by the equivalence relation $$u \sim u' \Leftrightarrow
v(u-u')> \gamma .$$ Note that $$M_{\gamma}(a)=\vert M_{\geq \gamma}/M_{>\gamma} \vert .$$
\end{nota}

\begin{proposition}\label{plongementsature}
Let $ M $ and $ N $ be as before, $ A $ be a submodule of $ M $ such that
$ v(A) =v( M) $ and ${\bf f} = (f, f_v): A \to N $ an $ L$-embedding of $ A $ into $ N $. Suppose that $ N $ is $ \absolue{M}^{+}$-saturated. Then we can extend ${\bf f} $ to an $L$-embedding of $ M $ into $ N $.
\begin{proof}
Note first that we can extend ${\bf f}$ on $\hat{A}$ by  Proposition \ref{Achapeau}. So we assume $\hat{A}=A$ with image $B=\hat{B}$.  Take $ x \in M \setminus A $. It suffices to extend {\bf f} to  $ A \oplus  x.R$. In that case, by the above proposition, we can extend {\bf f} to the  divisible closure (unique by the fact that $ A \supseteq M_{tor} $) of $ A \oplus  x.R$. Thus, by transfinite induction, {\bf f} extends to $M$.

Set
 $$ p(Y):=\{\gamma_a = w(Y-b) \tq \gamma_a=f_v(v(x-a)) \mesp \text{and} \mesp
b = f (a), \ a \in A\}.$$

We first show that if $ y \in N $  realizes $p(Y)$, then we can extend {\bf f} to an $ L$-embedding $:A \oplus x.R \to N $ sending $ x $ to $y $. Take  $y \in N $
realizing $ p (Y) $. Then  for all $ r \in R \setminus
\{0 \} $, $ y.r \notin B $. We set $ \tilde{f}(a + x.r) = f (a) +  y.r.$
Since $ v (A) = v (M) $ it suffices to see that $ \tilde{f} $ is compatible
with $f_v $,  i.e. to check that one has $f_v(v(a - x.r)) = w(f(a) - y.r)$, for all $ a \in A $ and $ r \in R \setminus \{0 \}$.

\noindent
{\it Claim}: For $a\in A$ and $r\neq 0$, $x-a$ is regular for $r$ if and only if  $y-f(a)$ is regular for $r$.

\noindent
{\it Proof of the Claim}. If $y-f(a)$ is irregular for some $r$, then for some $b \in  ann_N(r) \subseteq B$, $$w(y-f(a)-b)>\theta$$
by \ref{ConsT} (Consequence 2). Hence $v(x-a-f^{-1}(b))>\theta$ by the choice of $y$. This means $x-a$ is irregular by \ref{ConsT}. The inverse can be proven in the same way.

Let  $ a \in A $, $ r \in R \setminus \{0 \} $. Take $a'$ such that
$x.r - a=(x-a').r$  and $x-a'$ regular for $r$.
Hence $f(a)=f(a').r$ and $f(a')-y$ is regular for $r$ by the Claim above. Since $f_v$ commutes with $\tau$ we have $f_v(v(a - x.r)) = w(f(a) - y.r)$ as required.$\dagger$

It remains to prove that $p(Y) $ is realized in $ N $. We will show that it is finitely consistent. Let $ \alpha = \{a_1, \dots, a_m \} \subseteq A $.
We will find $ y \in N $ such that for all $ 1 \leq i \leq m$, $ \gamma_{i}: = f_v (v (x-a_i)) $ is equal to $ w (y - f (a_i)) $. Let $ \sim $ be the equivalence relation on $ \alpha $ defined by $ a_i \sim a_j $ if and only if $ \gamma_i = \gamma_j$. We will first observe that,  without loss of generality, we can assume that there is only one class under this equivalence relation: In fact, it is enough to consider $ \beta \subseteq \alpha $ be the equivalence class  of  $ \gamma: = \max \{\gamma_i \} _{i \in \{1, \dots, m \}} $. If $ a \in \beta $ and $ a' \notin \beta $ then \begin{multline*}
 f_v(v(x-a')) = f_v(v(x-a + a-a')) = f_v(v (a-a ')) = w(f (a) - f (a' )) \\ = w(f (a)-y + y-f(a')) = w(f(a')-y).\end{multline*}

Set $ \delta := f_v^{-1}(\gamma)$, with the assumption above, it is the common value of the $v(x-a)$ for $ a \in \alpha $. Now,  we can choose an element $ y \in N $ such that for all $a\in \alpha$
$$w(y-f(a))=\gamma$$  since
  $ \absolue{M_{\delta} (a)} \leq
\absolue{N_{\gamma} (f (a))} $, $ f_v $ preserves the predicate $R_n $ and
$ N $ is $\absolue{M}^+$-saturated.
\end{proof}
\end{proposition}

To prove Theorem \ref{aketrivial} we recall the following general result.

\begin{proposition}\label{qe}
Let $\mathcal {L} $ be a language containing a constant symbol, $\mathcal {T} $ an
$\mathcal {L} $-theory and $\Theta $
a set of formulas of $\mathcal {L} $ closed under boolean combinations and containing all quantifier free formulas. Suppose that for all
$ M, N\models\mathcal {T} $ where $ N $ is
$\absolue{M}^{+}$-saturated and for all substructure
$ A\subseteq M$, whenever $f: A \to N $  is an $\mathcal {L} $-embedding preserving formulas of
$\Theta $, there exists an $\mathcal {L} $-embedding $ g:M\to N $ which extends $f$ and preserves $\Theta $.
 Then every formula of $\mathcal {L} $ is equivalent to a formula of
$\Theta $.
\begin {proof}
If $\Theta $ is the set of  all quantifier-free formulas this is a well known fact (cf.
\cite {marker} proposition 4.3.28). The general case reduces to this case after adding to the language the predicates $P_{\phi}$ for each $\phi \in \Theta$ and enriching the theory $\mathcal{T}$ by the set of sentences
  $\{\forall x\mesp P_ {\phi} (x)\iffs \phi(x)\tq\phi\in\Theta \} $.
\end {proof}
\end{proposition}

Let $ (F, v (F))$ be a non zero henselian divisible $R$-module. Let $T_v $ be  the complete theory of
$ v (F) $ in the language $ L_{V}$ and $ \text{Tor}_F $ be the
$L_{Mod}(R)$-theory of non zero divisible $R$-modules  with  the extra  statements of the form ${\eta_F (r)} = n $ or sets of statements expressing $ \absolue{\eta_F (r)} = \infty $ where
$${\eta_F (r)}=\absolue{\{x\in F \tq x.r=0\}}.$$ By \ref{completions} this is the complete theory of (the pure module)  $F$ and also of $F_{tor}$ if it is non zero (by \ref{Mtordivisible}). Similarly to the case of valued fields, $\text{Tor}_F$ can be seen as {\it the residual theory} of $(F,v)$ since $F_{tor}$ embeds in $F_{\geq \theta}/F_{>\theta}$.

We set  {\bf T}$:= T_{h} \cup T_v \cup \text{Tor}_F $ and $\Theta_{v} $
the set of  $L$-formulas of the form:
$$ \varphi (\bar{x})  \wedge Q \bar{y_1}
\mesp \psi(\bar{y_1}, \bar{y_2}, v (t_1(\bar{x})), \dots, v(t_k(\bar{x})),$$
where:

-- $T_h$ is the theory of henselian valued $R$-modules,

-- $\varphi$ is quantifier-free in $L_{Mod}(R)$, $ \psi $ is quantifier-free in $L_{V}$ and $ Q $ is a sequence of quantifiers over $\bar{y}_1$,

-- $\bar{x} $ is a tuple of variables of the module sort, $\bar{y_1}, \bar{y_2}$ are tuples of variables of the  value set sort, and the $ t_i $ are $L_{Mod}(R)$-terms.

Theorem \ref{aketrivial} follows now by the following theorem:

\begin{theorem}\label{aketrivialQ}
The theory {\bf T}  is complete and eliminates quantifiers  on the  module sort: any formula of $ L$ is equivalent modulo {\bf T} to a formula of $ \Theta_v$.
\begin{proof}
Take $(M,v), (N,w) \models \mathbf{T}$ such that $N$ is $|M|^{+}$-saturated.
Let  $ {\bf f} = (f, f_v): (A,\Delta)  \to N $ be an $L$-embedding preserving formulas of $\Theta_v$ where $(A,\Delta)$ is a substructure of $(M,v)$. Note that  the condition $ \theta \in v(M) $  is described by the $L_{V_0}$-theory  of $v(F)$. Furthermore, since
$ f_v$  preserves in particular all $ L_{V}$-formulas, $f_v$ is partial elementary and extends to an elementary embedding of $v(M)$ into
$v(N)$ by saturation hypothesis. This extension will noted as $f_v$ as well. Hence we may assume $\Delta = v(M)$ (but $v$ is not necessarily surjective).

On the other hand by Proposition \ref{Achapeau}, there is an $ L$-embedding $ {\bf \hat {f}} = (\hat {f}, \hat {f_v}) $ extending $ {\bf f} $ to
$(\hat{A}, v(M))$. Note that
any element $ \gamma \in v(\hat {A}) \setminus \{\theta \} $ is $ v (A) $-definable by a formula of the form
$\tau^k (\gamma) = \gamma_a $ for some integer $k$ and some $ \gamma_a \in v(A) $. This implies that $ \mathbf{\hat{f}} $ preserves the formulas from $ \Theta_v$. Thus from now on we can assume that $ A = \hat {A} $.

-- {\bf Extending $f$ to an $(U,v)$ such that $v(U)=v(M)$.}
Let $\gamma \in v(M) \setminus A$, $x \in M$ of valuation $\gamma$ and $ y \in N $ such that $ w (y) = g_v (\gamma) $.  Since $ v (A) $ is closed by $ \tau $ and $ \tau^{-1} $, we have, for all non zero $ r \in R $, $ v(x.r) \notin v (A) $. It follows that $ v (a + x.r) = \min \{v (a), v (x.r) \} $  for all $ r \in R $ and  $a  \in A$.  Note that $x$ is regular. In fact, either $M$ is torsion-free and hence  $M$ is regular by Lemma \ref{tfh_gives_reg}, or, $M_{tor} \subseteq A$ and $\gamma\neq \theta$. Thus $v(x.r)=\tau^k(v(x))$ for some integer $k$ depending only on $\gamma$. Taking into account that $f_v$ is an elementary function,  the map $g$ defined by
 $$ g (a + x.r):= f (a) + y.r  \quad  (r \in R , a \in A) $$
 yields that $ {\bf g} = (g, f_v) $ is a partial $L$-isomorphism between $( A \oplus x.R, v) $ and $(B \oplus  y.R,w)$    and $ {\bf g} $ preserves  the set  $\Theta_v$ by the above discussion. By Proposition \ref{Achapeau}$, {\bf g} $ extends to $\widehat{A + x.R}$. Hence, by using the same argument, we can extend $ {\bf g} $ to a model $ U $ such that  $ v (U) = v (M) $.

 Now, by Proposition \ref{plongementsature} we can extend ${\bf g}$ to $M$ and hence we get quantifier elimination up to $\Theta_v$ by Theorem \ref{qe}.

It follows that any $L$-sentence is equivalent modulo {\bf T} to an $L_V$-sentence. Since by the definition of
{\bf T} any two models have $L_V$-elementary equivalent value sets,  they are in fact $L$-elementary equivalent. This gives the completeness of
{\bf T} as required.
\end{proof}
\end{theorem}

\begin{corr} Theorem \ref{aketrivial} is a consequence of the completeness of ${\bf T}$. Moreover,
\begin{enumerate}[$\bullet$]
\item {\normalfont (A-K,E $ \preceq $):}\label{A-K-Epreceq}
Let $ (M \subseteq N, v) $ be an extension of non-zero henselian divisible valued modules such that the inclusion $(M_{tor} \subseteq N_{tor})$ is $L_{Mod}(R)$-elementary and  the inclusion ($v(M) \subseteq v(N)$) is $L_V$-elementary. Then the inclusion ($M \subseteq N, v$) is $L$-elementary.
\end{enumerate}
\end{corr}
\bibliographystyle{plain}
\bibliography{valmodgo}

\end{document}

\end{document}